\newif\ifpictures
\picturestrue

\documentclass[12pt]{amsart}
\usepackage{amssymb}
\usepackage[dvips]{graphicx}
\usepackage{color}

\usepackage{epsfig}

\numberwithin{equation}{section}
\newtheorem{thm}{Theorem}
\newtheorem{prop}[thm]{Proposition}
\newtheorem{lem}[thm]{Lemma}
\newtheorem{cor}[thm]{Corollary}

\theoremstyle{definition}

\newtheorem{example}[thm]{Example}
\newtheorem{remark1}[thm]{Remark}
\newtheorem{openproblem1}[thm]{Open problem}

\newenvironment{ex}{\begin{example}\rm}{\end{example}}

\numberwithin{thm}{section}

\newcounter{FNC}[page]
\def\newfootnote#1{{\addtocounter{FNC}{2}$^\fnsymbol{FNC}$%
     \let\thefootnote\relax\footnotetext{$^\fnsymbol{FNC}$#1}}}

\newcommand{\C}{\mathbb{C}}
\newcommand{\N}{\mathbb{N}}

\newcommand{\R}{\mathbb{R}}

\DeclareMathOperator{\SO}{SO}
\DeclareMathOperator{\Sym}{S}
\DeclareMathOperator{\id}{id}
\DeclareMathOperator{\rad}{rad}

\begin{document}

\title{Determining a rotation of a tetrahedron from a projection}
\author[Richard J. Gardner, Paolo Gronchi, and Thorsten Theobald]
{Richard J. Gardner, Paolo Gronchi, and Thorsten Theobald}
\address{Department of Mathematics, Western Washington University,
Bellingham, WA 98225-9063} \email{Richard.Gardner@wwu.edu}
\address{Dipartimento di Matematica ``Ulisse Dini", Universit\`a degli Studi di Firenze, Piazza Ghiberti 27, 50122
Firenze, Italy} \email{paolo@fi.iac.cnr.it}
\address{FB 12 -- Institut f\"{u}r Mathematik, Goethe-Universit\"{a}t, Postfach 111932, D-60054 Frankfurt am Main, Germany}
\email{theobald@math.uni-frankfurt.de}
\thanks{First author supported in
part by U.S.~National Science Foundation Grants DMS-0603307 and DMS-1103612. Third author supported in part by DFG Grant TH 1333/2-1.}
\subjclass[2010]{Primary: 52B10, 68U05, 68W30; secondary: 05E18}
\keywords{perspective-$n$-point problem, reconstruction, tetrahedron,
  medical imaging, geometric tomography}

\begin{abstract}
The following problem, arising from medical imaging, is addressed:  Suppose that $T$ is a known tetrahedron in $\R^3$ with centroid at the origin. Also known is the orthogonal projection $U$ of the vertices of the image $\phi T$ of $T$ under an unknown rotation $\phi$ about the origin. Under what circumstances can $\phi$ be determined from $T$ and $U$?
\end{abstract}

\date{\today}

\maketitle
\section{Introduction}

The \emph{perspective-$n$-point} problem, often abbreviated P$n$P, is the problem of
determining the position of a camera from the perspective images of $n$ given points. The problem has been widely investigated during the last few decades, using several traditional camera models, such as projective (see, for example, \cite{hcl-89}),
orthographic  (see, for example, \cite{huttenlocher-ullman-87}),
or weak perspective (i.e., scaled orthographic, see \cite{alter-94,hbh-95}),
and focusing on various aspects (such as small values of $n$).

While the solution of a specific instance of P$n$P is often an application of
elementary geometry, understanding the configuration space---for example, classifying which configurations admit a given number of solutions---involves challenging nonlinear aspects
(cf.\ \cite{fmr-2008, rieck-2011} and the references therein).
 Indeed, it was not until recently that Faug\`{e}re et al \cite{fmr-2008}
(partially) classified the configurations for the perspective-3-point problem via the discriminant variety, using extensive computations.

Our point of departure is a paper by Robinson, Hemler and
Webber \cite{rhw-2000}, who, motivated by an application in
imaging, studied the perspective-4-point problem for the orthographic camera model.  The problem is as follows.  A given tetrahedron $T$ in $\R^3$ with
vertices $p^{(1)}, \ldots, p^{(4)}$ has been transformed by an unknown (direct) rigid motion $\phi$. Also given is the image $U = \{u^{(1)}, \ldots, u^{(4)}\}$ of the set of vertices of $\phi T$ under a parallel projection onto the $xy$-plane in an unknown direction $w\in S^2$.  The problem is to find $\phi$ and $w$.

In \cite{rhw-2000} it is observed that one may as well take the parallel projection to be the orthogonal projection $\pi_z$ onto the $xy$-plane.  It is also noted that then $\phi$ can only be determined up to a vertical translation, because such a translation does not change $U$.  Since $\phi$ is the composition of a rotation about the origin and a translation, it suffices to determine the rotation and the horizontal component of the translation. The authors of \cite{rhw-2000} make the assumption that it is known which projection comes from which vertex of $T$, that is, they assume that $u^{(i)} = \pi_z\phi p^{(i)}$, $i=1,\dots,4$. Under this labeling assumption, they show that the rotation and horizontal shift can be determined.

Our purpose here is to study this problem when the labeling assumption is removed, and to provide a systematic foundational study from the viewpoint of nonlinear computational geometry (see, for example, \cite{ast-2011,est-2009,petitjean-99}).

Clearly, the centroid of the vertices of $\phi T$ must lie on the vertical line through the known centroid of $U$.  From this, we make two conclusions.  Firstly,  the horizontal shift can always be determined, so we may assume that $\phi$ is a rotation about the origin. Secondly, if such a rotation $\phi$ can be determined when the centroid of $T$ is at the origin, then it can also be determined when the centroid of $T$ is located elsewhere.  Thus our problem can be stated in the following form.

\medskip

\noindent Suppose that $T$ is a \emph{known} tetrahedron in $\R^3$ with
vertices $p^{(1)}, \ldots, p^{(4)}$ and centroid at the origin. Also \emph{known} is the orthogonal projection $U = \{u^{(1)}, \ldots, u^{(4)}\}$ onto the $xy$-plane of the vertices of the image $\phi T$ of $T$ under an \emph{unknown} rotation $\phi$ about the origin. Under what circumstances can we determine $\phi$ from $T$ and $U$?

\medskip

Obviously, if $T$ has nontrivial automorphisms---for example, if $T$ is regular---then $\phi$ cannot be uniquely determined. Now let $T$ be an arbitrary tetrahedron in $\R^3$ with vertices $p^{(1)}, \ldots, p^{(4)}$.  Suppose that the images $\phi\left((p^{(1)} + p^{(2)})/2\right)$
and $\phi\left((p^{(3)} + p^{(4)})/2\right)$ under $\phi$ of the midpoints of the opposite edges $[p^{(1)},p^{(2)}]$ and $[p^{(3)},p^{(4)}]$
are contained in the $z$-axis. Then a rotation $\psi$ of $\phi T$ by $\pi$
about the $z$-axis results in a tetrahedron $\psi \phi T$ whose vertices also project onto $U$. In this case $U$ forms the vertices of a parallelogram in the $xy$-plane, so $U$ has a symmetry (rotation by $\pi$ about its center).

These preliminary remarks show that in general $\phi$ cannot be determined
if $T$ or $U$ has extra symmetries. A general goal is to understand if it can
be uniquely determined otherwise, and if not, to find those $T$ and $U$
that do allow $\phi$ to be determined.

The relation between our problem and the one considered in \cite{rhw-2000} can be made clearer if we regard the labels of the vertices of $T$ as having been permuted by an unknown permutation $\sigma$ of $\{1,2,3,4\}$, so that $u^{(i)}$ is the projection of $\phi p^{(\sigma(i))}$, $i=1,\dots,4$.  Then the problem in \cite{rhw-2000} corresponds to the case when $\sigma$ is the identity.

In this paper, we deal with both uniqueness and reconstruction.
Our focus is on the geometry of the problem, in particular, the configuration space of all tetrahedra leading (for a given rotation) to the same set of projection points as the original tetrahedron. By decomposing this space into the
union of the spaces corresponding to the various types of
permutations involved, we can treat the configuration questions
from a linear algebra point of view. Then, using some
nonlinear symbolic methods, we precisely classify situations where the dimension of the configuration space deviates from the expected dimension.  As a consequence, we are able to prove in Theorem~\ref{thm0} that for almost all tetrahedra $T$ in $\R^3$ with centroid at the origin, there does not exist a rotation $\phi$ other than the identity such that $\pi_z\phi T=\pi_zT$.  However, the various lemmas that we prove along the way provide much more detailed information.

Problems such as the one addressed here, involving the retrieval of information about a geometric object from data concerning its projections onto lines or planes (or intersections with lines or planes), fall under the umbrella of geometric tomography \cite{gardner-2006}.

The paper is structured as follows. After the preliminary Section~\ref{se:prelim}, the case when the permutation $\sigma$ is the identity is considered in Section~\ref{se:labeled}
from a linear algebra and symbolic viewpoint. Then, in Sections~\ref{se:onec},~\ref{se:twoc},~\ref{se:threec}, and~\ref{se:fourc},
we deal with the other cases. Finally, in Section~\ref{se:main}
we state the main conclusions for our study.

\section{Notation and preliminaries\label{se:prelim}}

As usual, $S^{n-1}$ denotes the unit sphere and $o$ the origin in Euclidean
$n$-space $\R^n$.  The Euclidean norm is denoted by $\|\cdot\|$. Unless specified otherwise, $x_i$ will signify the $i$th coordinate of a point $x=(x_1,\dots,x_n)$ in $\R^n$.  The unit ball in $\R^n$ will be denoted by $B^n$.  We write
$[x,y]$ for the line segment with endpoints $x$ and $y$. Orthogonal projection onto the $xy$-plane in $\R^3$ is denoted by $\pi_z$.  Given $u\in S^2$, we denote the line through the origin parallel to $u$ by $l_u$.  The {\em dimension} $\dim A$ of a set $A$ in $\R^n$ is the dimension of its affine hull.
The {\em symmetric group} on $\{1,2,\dots,n\}$ is denoted by $\Sym_n$.

Let $\SO(3)$ denote the {\em group of rotations} about the origin in $\R^3$.
An element of $\SO(3)$, henceforth simply called a {\em rotation}, can be specified in terms of a rotation axis (a line through $o$) and a rotation angle.  We shall also use the following characterization using quaternions (see, for example, \cite[Sec.~8.2]{gallier-2001}). For a quaternion $q = a + bi + cj + dk$, where $a,b,c,d\in\R$, the rotation matrix $R(q)$ associated with $q$ is
\begin{equation}
  \label{eq:quaternionmatrix}
  R(q)=
  \frac{1}{\|q\|^2}
  \left( \begin{array}{ccc}
    a^2+b^2-c^2-d^2 & 2bc - 2ad & 2 bd + 2ac \\
    2bc + 2ad & a^2 - b^2 + c^2 - d^2 & 2cd - 2ab \\
    2bd - 2ac & 2cd + 2ab & a^2 - b^2 - c^2 + d^2
  \end{array} \right),
\end{equation}
where $\|q\|=\sqrt{a^2+b^2+c^2+d^2}$. Conversely, the quaternion $q$ corresponding to a rotation with axis in the direction $w=(w_1,w_2,w_3)\in S^2$ and rotation angle $\alpha$ is
\begin{equation}
  \label{eq:quaternionaxis}
  q=\cos(\alpha/2)+ w_1 \sin (\alpha/2)i+ w_2 \sin(\alpha/2)j+ w_3 \sin(\alpha/2)k.
\end{equation}
Since a rotation around an axis of rotation $l_u$ by angle $-\alpha$ is the same as a rotation around $l_{-u}$ by angle $\alpha$, we may without loss of generality restrict $\alpha$ to the interval $[0,\pi]$.

It will be convenient to regard a {\em tetrahedron} in $\R^3$ simply as a set of four points in $\R^3$.  Either these points are in general position, in which case they form the set of vertices of a full-dimensional tetrahedron in the usual sense of the term, or they are contained in a plane and hence lower dimensional.

Throughout, we consider a known tetrahedron $T=\{p^{(1)}, \ldots, p^{(4)}\}$ in $\R^3$ with centroid at the origin.  The projection $U = \{u^{(1)}, \ldots, u^{(4)}\}$ of $\phi T$ onto the $xy$-plane, where $\phi\in \SO(3)$ is unknown, is also given.  Then there is an unknown permutation $\sigma\in \Sym_4$ such that
\begin{equation}
  \label{eq:prelim1}
 \pi_z\phi p^{(i)}= u^{\sigma(i)},
\end{equation}
for $i=1,\dots,4$.

The case dealt with in \cite{rhw-2000}, corresponding to $\sigma=\id$, the identity permutation, can be viewed as that of a {\em labeled tetrahedron}; the projections of the vertices retain the labels, so that it is known which point in $U$ corresponds to which vertex of $T$.

If $\phi\in \SO(3)$ and $\sigma\in \Sym_4$, we denote by $\mathcal{T}_{\sigma}(\phi)$ the family of (possibly lower-dimensional) tetrahedra $T=\{p^{(1)}, \ldots, p^{(4)}\}$ such that
\begin{equation}\label{eq:cond2}
(\phi p^{(i)})_j=p^{({\sigma(i)})}_j,
\end{equation}
for $i=1,\dots,4$ and $j=1,2$.  When there are two different rotations of $T$ giving rise to the same set $U$ of projections onto the $xy$-plane, we may for our purposes assume that one rotation is the identity, and then by (\ref{eq:prelim1}), (\ref{eq:cond2}) holds for some $\phi\neq\id$.  Thus it suffices to study the system (\ref{eq:cond2}) in order to understand uniqueness issues, and we will be interested in the dimension of $\mathcal{T}_{\sigma}(\phi)$ in various situations.
Since each tetrahedron is a set of four points in $\R^3$, we could also regard a tetrahedron as a point in $(\R^3)^4\simeq\R^{12}$, and thus consider $\mathcal{T}_{\sigma}(\phi)$ as a set in $\R^{12}$.  However, we are assuming that $T$ has centroid at the origin, so that
\begin{equation}\label{cendep}
\sum_{i=1}^4 p^{(i)}=o=\sum_{i=1}^4 \phi p^{(i)}.
\end{equation}
Using (\ref{cendep}), we may identify $T$ with any three of its points, say the first three, and then the equation in (\ref{eq:cond2}) corresponding to $i=4$ is redundant.  We shall therefore identify $\mathcal{T}_{\sigma}(\phi)$ with the corresponding set in $\R^{9}$, which, in view of (\ref{eq:cond2}) and (\ref{cendep}), is actually a subspace of $\R^9$. Of course each tetrahedron gives rise to not one but 24 points in $\R^{9}$ (depending on which three of its vertices are selected and in which order), but since we are only interested in the dimension of $\mathcal{T}_{\sigma}(\phi)$, this loss of bijectivity is unimportant.

Clearly there are only five essentially different cases to consider.  There is the labeled case when $\sigma=\id$, and if $\sigma\neq \id \in S_4$, then $\sigma$ is a two-cycle, a direct product of two two-cycles, a three-cycle, or a four-cycle.  Corresponding to the four latter cases, we can, without loss of generality, consider in turn (i) $\sigma=(2,1,3,4)$, (ii) $\sigma=(2,1,4,3)$, (iii) $\sigma=(2,3,1,4)$, and (iv) $\sigma=(2,3,4,1)$.

\section{The labeled case: $\sigma=\id$\label{se:labeled}}

In the terminology introduced in the previous section, Robinson, Hemler,
and Weber \cite{rhw-2000} proved the following result.

\begin{prop}\label{PHW}
Let $T$ be a full-dimensional labeled tetrahedron in $\R^3$. Then there do not exist two different rotations such that the resulting (labeled) projections of the rotated vertices of $T$ onto the $xy$-plane coincide. Thus the rotation is uniquely determined by the (labeled) projection.
\end{prop}

\begin{proof}
Suppose two different rotations as in the statement of the proposition exist.  Clearly we may assume that one is the identity $\id$ and denote the other by $\phi\neq \id$.  If the resulting projections coincide, then from (\ref{eq:cond2}) with $\sigma=\id$, we obtain
\begin{equation}
  \label{eq:cond1}
  (\phi p^{(i)})_j=p^{(i)}_j,
\end{equation}
where $i=1,\dots,4$ and $j=1,2$.  Let
\begin{equation}
  \label{eq:hyperplane1}
  H_j=\{x \in \R^3:(\phi x)_j - x_j = 0\},
\end{equation}
for $j=1,2$.  Then $p^{(i)}\in H_1\cap H_2$ for $i=1,\dots,4$. If either of the subspaces $H_1$ and $H_2$ are proper subsets of $\R^3$, we are done, since $T\subset H_1\cap H_2$ contradicts the assumption that $T$ is full dimensional.  Otherwise, we have $(\phi x)_j - x_j = 0$ for all $x\in \R^3$ and $j=1,2$. But then $\phi$ fixes the $xy$-plane and hence $\phi=\id$.
\end{proof}

The authors of \cite{rhw-2000} gave a different proof of the previous proposition, deriving it from a reconstruction procedure. For the convenience of the reader, we provide a different reconstruction method that can be obtained from the reconstruction
result for the perspective-3-point problem under weak perspective (see \cite{alter-94,hbh-95}). Extending the reconstruction algorithm in \cite{hbh-95}
to the four-point case works as follows.

Let $\phi$ be the unknown rotation. We first construct the unique circle $C$ containing the known points $p^{(i)}$ for $i=1,2,3$.  We aim to construct the projection $E=\pi_z\phi C$ of $\phi C$, an ellipse in the $xy$-plane whose semi-major axis has length equal to the radius of $C$.  The known points
$u^{(i)}$ for $i=1,2,3$ lie on $E$. For $i=1,2,3$, denote by $m^{(i)}$ the midpoint of the edge of the triangle $p^{(1)}, p^{(2)}, p^{(3)}$ opposite to $p^{(i)}$ and by $t^{(i)}$ the other intersection of the line through
$p^{(i)}$ and $m^{(i)}$ with the circle $C$.  The corresponding midpoints $\pi_z \phi m^{(i)}$ of the edges of the triangle $u^{(1)}, u^{(2)}, u^{(3)}$ opposite to $u^{(i)}$ can of course be constructed since this triangle is known.
Then, for $i=1,2,3$, the point $\pi_z \phi t^{(i)}$ can be constructed by elementary geometry, since
$$\frac{\|p^{(i)}-m^{(i)}\|}{\|p^{(i)}-t^{(i)}\|} =\frac{\|u^{(i)}-\pi_z\phi m^{(i)}\|}{\|u^{(i)}-\pi_z\phi t^{(i)}\|},$$
for $i=1,2,3$.  Since $\pi_z \phi t^{(i)}$ lies on $E$ for $i=1,2,3$, we have constructed six points on $E$.  But any five points determine an ellipse, so we can construct $E$ itself. Now $E$ determines the circle $\phi C$, up to reflection in the $xy$-plane and vertical translation, and hence the points
$\phi p^{(1)}$, $\phi p^{(2)}$, and $\phi p^{(3)}$ are similarly determined.
Since $T$ is known, the position of $\phi p^{(4)}$ is also known relative to $\phi C$, up to a reflection in the plane containing $\phi C$.  If $T$ is full dimensional, $\det\phi$ is determined by the points $\phi p^{(i)}$, $i=1,\dots,4$, and since $\det\phi=1$, no reflection is possible.
Now we can use the fact that because the centroid of $T$ is at the origin, the centroid of $\phi T$ is also.  This allows $\phi T$ and hence (since we are in the labeled case) $\phi$ to be completely determined, if $T$ is full dimensional, and up to a reflection in the $xy$-plane, if $T$ is contained in a plane.  Note that if $T$ is contained in a plane, a reflection of $T$ in the $xy$-plane is of the form $\psi T$ for some rotation $\psi$ about the origin, so $\phi$ cannot be fully determined in this case.

Recall that we regard the family $\mathcal{T}_{\sigma}(\phi)$ as a set in $\R^{9}$ and that we are considering the case $\sigma=\id$.

\begin{lem}\label{ident}
Let $\phi \neq \id$ be a rotation.  Then $\dim \mathcal{T}_{\id}(\phi)=3$, unless the axis of rotation is horizontal, when $\dim \mathcal{T}_{\id}(\phi)=6$.
\end{lem}

\begin{proof}
Let $T$ be a tetrahedron with vertices $p^{(i)}$, $i=1,\dots,4$ and centroid at the origin. Identifying $T$ with $p^{(i)}$, $i=1,2,3$ and using  \eqref{eq:cond2} with $\sigma=\id$ and (\ref{cendep}), we see that $T\in \mathcal{T}_{\id}(\phi)$ if and only if
\begin{equation}\label{eq:cond3}
(\phi p^{(i)})_j - p^{(i)}_j =0,
\end{equation}
for $i=1,2,3$ and $j=1,2$, a system of six equations in nine variables. Let $M$ be the corresponding $6 \times 9$ coefficient matrix,
where the variables are ordered $p^{(1)}_1$,
$p^{(1)}_2$, $p^{(1)}_3$, $p^{(2)}_1,\dots, p^{(3)}_3$, and where for $i=1,2,3$, rows $2i-1$ and $2i$ of $M$ correspond to the equations with index $j=1$ and $2$, respectively.  Then $\dim \mathcal{T}_{\id}(\phi)$ equals the dimension of the null space of $M$.

Since $M$ obviously has rank at most six, we obtain $\dim \mathcal{T}_{\id}(\phi)\ge 9-6=3$ directly from the Rank Theorem.

Let
\begin{equation}\label{AI}
A=\left( \begin{array}{ccc}
    a^2+b^2-c^2-d^2 & 2bc - 2ad & 2 bd + 2ac \\
    2bc + 2ad & a^2 - b^2 + c^2 - d^2 & 2cd - 2ab
  \end{array} \right)\quad{\text{and}}\quad I=\left( \begin{array}{ccc}
    1 & 0 & 0\\
    0 & 1 & 0
  \end{array} \right).
\end{equation}
The rotation $\phi$ can be represented by the matrix (\ref{eq:quaternionmatrix}) with $a^2+b^2+c^2+d^2=1$, and using the latter equation and (\ref{eq:cond3}), we can rewrite $M$ as a block matrix,
$$M=\left( \begin{array}{ccc}
    A-I & 0 & 0 \\
    0 & A-I & 0\\
    0 & 0 & A-I
  \end{array} \right),$$
where
$$
A-I=2\left( \begin{array}{ccc}
    -c^2-d^2 & bc - ad & bd + ac\\
    bc + ad & - b^2- d^2 & cd - ab
  \end{array} \right).
$$

Suppose that $\dim \mathcal{T}_{\id}(\phi)>3$. Then the rank of $M$ is less than six, so all the $6 \times 6$ minors of $M$ vanish. The $6\times 6$ minor corresponding to columns 1, 2, 4, 5, 7, and 8 of $M$ is
$$8\det\left( \begin{array}{cc}
    -c^2-d^2 & bc - ad\\
    bc + ad & - b^2- d^2
  \end{array} \right)^3=8d^6(a^2+b^2+c^2+d^2)^3=8d^6.$$
Hence $d=0$, which in view of (\ref{eq:quaternionaxis}) implies that the axis of rotation is horizontal.  From the geometry it is clear that without loss of generality, we may suppose that this axis is parallel to $(1,0,0)$, so that $b=\sin(\alpha/2)\neq 0$ and $c=0$.  But then
$$A=2\left( \begin{array}{ccc}
    0 & 0 & 0\\
    0 & - b^2 &- ab
  \end{array} \right),$$
in which case all $4\times 4$ minors vanish but not all $3\times 3$ minors do so. Then the rank of $M$ is three, so $\dim \mathcal{T}_{\id}(\phi)=9-3=6$.
\end{proof}

The geometry corresponding to the previous lemma is as follows.  We know from Proposition~\ref{PHW} that each member of $\mathcal{T}_{\id}(\phi)$ is degenerate, and hence contained in a plane.  If $\phi$ is a rotation about a line not contained in the $xy$-plane, then the only solutions to (\ref{eq:cond3}) are those for which each point $p^{(i)}$, $i=1,2,3$, is contained in the axis of rotation.  For each $p^{(i)}$ there is one degree of freedom, and hence the set of solutions is three dimensional.  Suppose, on the other hand, that $\phi$ is a rotation by angle $\alpha$ about a horizontal line.  Then the points $p^{(i)}$, $i=1,2,3$, must lie in one of the two planes containing this line and at an angle $\alpha/2$ to the $xy$-plane.  For each $p^{(i)}$ there are two degrees of freedom, so the set of solutions is six dimensional.

Note that in the previous discussion about $\dim\mathcal{T}_{\id}(\phi)$, the position of $p^{(4)}$ is determined by the centroid condition (\ref{cendep}), once the positions of $p^{(i)}$, $i=1,2,3$, are known.  We shall use this fact frequently in the sequel without special mention.

The geometric statements in Proposition~\ref{PHW} and Lemma~\ref{ident} yield an algebraic corollary. To formulate this, let $I$ be the ideal in the real polynomial ring
$$R=\R[p^{(1)}_1, p^{(1)}_2, p^{(1)}_3, p^{(2)}_1, \ldots, p^{(4)}_3]$$
generated by the linear polynomials
\begin{equation}
\label{eq:poly}
  (\phi p^{(i)})_j - p^{(i)}_j,
\end{equation}
for $i=1,\dots,4$ and $j=1,2$.  An ideal generated by linear forms is also called a \emph{linear ideal}.

\begin{cor}\label{cor4}
For any rotation $\phi\neq \id$, a positive power of
the polynomial
\begin{equation}
  \label{eq:det}
  \det \left( \begin{array}{cccc}
    1 & 1 & 1 & 1 \\
    p^{(1)}_{1} & p^{(2)}_{1} & p^{(3)}_{1} & p^{(4)}_{1} \\
    p^{(1)}_{2} & p^{(2)}_{2} & p^{(3)}_{2} & p^{(4)}_{2} \\
    p^{(1)}_{3} & p^{(2)}_{3} & p^{(3)}_{3} & p^{(4)}_{3}
  \end{array}
  \right)
\end{equation}
is contained in the linear ideal $I$.
\end{cor}

\begin{proof}
By Proposition~\ref{PHW}, whenever a sequence of points
$(p^{(1)}, \ldots, p^{(4)})$ is a zero of the polynomials~\eqref{eq:poly} then the
points $p^{(1)}, \ldots, p^{(4)}$ are affinely dependent, that is, the determinant~\eqref{eq:det} vanishes. This determinant can be
seen as a polynomial in $p^{(i)}_j$'s. Thus, by the weak form of Hilbert's Nullstellensatz (see, for example, \cite[Section~4.1]{clo-iva}), this determinant polynomial is contained in the radical ideal $\rad(I)=\{r\in R: r^n\in I {\mathrm{~for ~some~}} n\in \N\}$ .

We remark that though the Nullstellensatz is a statement over the complex numbers, standard Gr\"obner basis theory implies that the determinant is also contained in the real linear ideal. Namely, since the ideal is generated by real polynomials, the standard algorithms for computing a Gr\"obner basis of $\rad(I)$ (see, for example, \cite[Theorem~8.99]{becker-weispfenning-93}) always keep coefficients within the reals and thus provide a real basis for $\rad(I)$. Similarly, the algorithm for reducing a real polynomial with respect to a Gr\"obner basis generated by real polynomials keeps coefficients within the reals. Since for any polynomial in the ideal this reduction algorithm yields a representation in terms of the generators, our remark follows.
\end{proof}

\section{One two-cycle: $\sigma=(2,1,3,4)$\label{se:onec}}

\begin{lem}\label{onec}
Let $\phi \neq \id$ be a rotation by angle $\alpha$ and let $\sigma=(2,1,3,4)$.  Then $\dim \mathcal{T}_{\sigma}(\phi)=3$, unless ${\mathrm{(i)}}$ the axis of rotation is neither horizontal nor vertical and $\alpha=\pi$, when $\dim \mathcal{T}_{\sigma}(\phi)=4$, or ${\mathrm{(ii)}}$ either the axis of rotation is vertical and $\alpha=\pi$ or the axis of rotation is horizontal and  $0<\alpha<\pi$, when $\dim \mathcal{T}_{\sigma}(\phi)=5$, or ${\mathrm{(iii)}}$ the axis of rotation is horizontal and $\alpha=\pi$, when $\dim \mathcal{T}_{\sigma}(\phi)=6$.
\end{lem}

\begin{proof}
Let $T$ be a tetrahedron with vertices $p^{(i)}$, $i=1,\dots,4$ and centroid at the origin. Identifying $T$ with $p^{(i)}$, $i=1,2,3$ and using \eqref{eq:cond2} and (\ref{cendep}), we see that $T\in \mathcal{T}_{\sigma}(\phi)$ if and only if
\begin{equation}\label{eq:cond3a}
(\phi p^{(1)})_j - p^{(2)}_j=0,\quad
(\phi p^{(2)})_j - p^{(1)}_j=0,\quad{\text{and}}\quad
(\phi p^{(3)})_j - p^{(3)}_j=0,
\end{equation}
for $i=1,2,3$ and $j=1,2$. Let $M_1$ be the $6 \times 9$ coefficient matrix of this system of six equations in nine variables, where the variables are ordered $p^{(1)}_1$, $p^{(1)}_2$, $p^{(1)}_3$, $p^{(2)}_1,\dots, p^{(3)}_3$, and where for $i=1,2,3$, rows $2i-1$ and $2i$ of $M_1$ correspond to the equations with index $j=1$ and $2$, respectively.  Then $\dim \mathcal{T}_{\sigma}(\phi)$ equals the dimension of the null space of $M_1$.

From (\ref{eq:cond3a}), we have
$$M_1=\left( \begin{array}{ccc}
    A & -I & 0 \\
    -I & A & 0\\
    0 & 0 & A-I
  \end{array} \right),$$
where $A$ and $I$ are given by (\ref{AI}).

Since $M_1$ obviously has rank at most six, we obtain $\dim \mathcal{T}_{\sigma}(\phi)\ge 9-6=3$ directly from the Rank Theorem.

Let $J_1$ be the ideal generated by all $6 \times 6$-minors of $M_1$ together with
the polynomial $\tau=a^2+b^2+c^2+d^2-1$. A Gr\"obner basis $G_1$ of $J_1$ with respect to the lexicographic ordering $a \succ b \succ c \succ d$ is given by
\begin{eqnarray*}
G_1&=&\{a^2d^4,a^2c d^3, a^2c^2 d^2, a^2b d^3, a^2b c d^2, a^2d^2 (b^2 - c^2 - d^2), a d^4 (d^2-1), a c d^3, a c^2 d^2, a b d^3,\\
& & a b c d^2,a b^2 d^2, \tau\}.
\end{eqnarray*}
(This can be found with a variety of standard software. Experts may well prefer a different choice, but with Mathematica, it can be done by defining the matrix {\tt $M_1$}, using {\tt Minors[$M_1$,6]} to generate the $6\times 6$ minors of $M_1$, adjoining the polynomial $a^2 +b^2 + c^2 + d^2-1$ to this list, and then using {\tt GroebnerBasis[\{list\},\{a,b,c,d\}]}.)
From this Gr\"obner basis we see that if $\dim \mathcal{T}_{\sigma}(\phi)>3$, then $a=0$ or $d=0$.  One can check that the rank of $M_1$ is $5$, $4$, or $3$, when $a=0$ and $d\neq 0,\pm 1$, or when either $a=0$ and $d=\pm 1$ or $a\neq 0$ and $d=0$, or when $a=d=0$, respectively.  This yields $\dim \mathcal{T}_{\sigma}(\phi)$ for cases (i), (ii), and (iii) in the statement of the lemma.
\end{proof}

In order to describe the geometry behind the previous lemma, for $i=1,2$, let $H_i$ be the plane containing $p^{(i)}$ which is orthogonal to the axis of rotation $l_u$, let $C_i$ be the circle in $H_i$ containing $p^{(i)}$ and with center on $l_u$, and let $l_i$ be the vertical line through $p^{(i)}$.  See Figure~\ref{fi:picture2}.  The rotation $\phi$ takes $p^{(1)}$ on $l_1$ around the circle $C_1$ to the point $\phi p^{(1)}$ on $l_2$ and also takes $p^{(2)}$ on $l_2$ around the circle $C_2$ to the point $\phi p^{(2)}$ on $l_1$. The angle of rotation is of course the same in each case, and we also have
$$\|p^{(1)}-\phi p^{(2)}\|=\|p^{(2)}-\phi p^{(1)}\|,$$
since the planes $H_1$ and $H_2$ are parallel and so intersect $l_1$ and $l_2$ in equidistant pairs of points.  It follows that $C_1$ and $C_2$ have equal radii and hence $p^{(1)}$ and $p^{(2)}$ are the same distance from $l_u$.

\begin{figure}[htb]
\epsfig{file=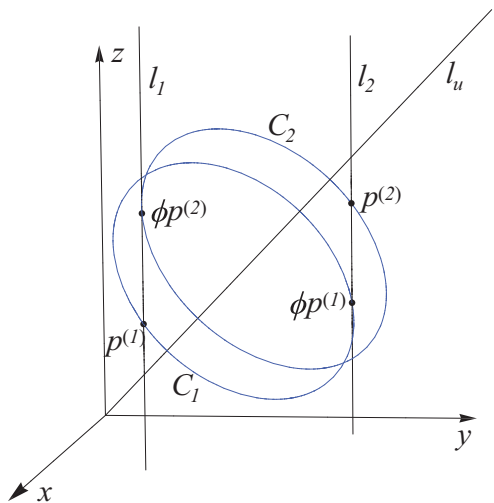, bb=7 7 239 237, clip=,
width=2.7in}
\caption{Geometry behind Lemma 4.1.}
\label{fi:picture2}
\end{figure}

If $l_u$ is neither vertical nor horizontal, then $\pi_zC_1$ and $\pi_zC_2$ are ellipses with their centers on $\pi_zl_u$.  If $\pi_zC_1\neq \pi_zC_2$, these two ellipses intersect in two points, namely $\pi_zp^{(1)}= \pi_z\phi p^{(2)}$ and $\pi_zp^{(2)}= \pi_z\phi p^{(1)}$, which are reflections of each other in $\pi_zl_u$.  Moreover, the angle $\alpha$ of rotation must be strictly between $0$ and $\pi$. Since $C_1$ and $C_2$ must intersect the vertical lines $l_1$ and $l_2$ through these two points, there is only one degree of freedom in choosing the position of each of $p^{(1)}$ and $p^{(2)}$.  The point $p^{(3)}$ must lie on $l_u$, allowing a further degree of freedom, so there are a total of three degrees of freedom, as in the first statement of the lemma.  If $\pi_zC_1= \pi_zC_2$, then $C_1=C_2$ and $\alpha=\pi$. In this case there are three degrees of freedom in choosing $p^{(1)}$, after which the position of $p^{(2)}$ is determined, and one in choosing $p^{(3)}$.  This corresponds to case (i) in the statement of the lemma.

Suppose that $l_u$ is the $z$-axis. Then $C_1$ and $C_2$ are possibly different horizontal circles and $\alpha=\pi$. There are three degrees of freedom in choosing $p^{(1)}$ and one each for $p^{(2)}$ and $p^{(3)}$, since the latter point must lie on the $z$-axis.  This situation is included in case (ii) in the statement of the lemma.

Finally, suppose that $l_u$ is contained in the $xy$-plane.  If $0<\alpha<\pi$, then $C_1=C_2$ is a vertical circle.  There are three degrees of freedom choosing $p^{(1)}$, after which the position of $p^{(2)}$ is determined, and two degrees of freedom in choosing $p^{(3)}$ (which must lie in the plane containing $l_u$ and having angle $\alpha/2$ with the $xy$-plane).  Again, this is included in case (ii) in the statement of the theorem. If $\alpha=\pi$, then $C_1$ and $C_2$ are possibly different circles lying in the same vertical plane.  There are three degrees of freedom in choosing $p^{(1)}$ and then one degree of freedom for $p^{(2)}$, since it can lie anywhere on the vertical line through $\phi p^{(1)}$. The point $p^{(3)}$ lies in the vertical plane containing $l_u$, allowing two degrees of freedom in its choice.  Thus there are six degrees of freedom in total, as in case (iii) in the statement of the lemma.

\section{Two two-cycles: $\sigma=(2,1,4,3)$\label{se:twoc}}

\begin{lem}\label{twoc}
Let $\phi \neq \id$ be a rotation by angle $\alpha$ and let $\sigma=(2,1,4,3)$.  Then $\dim \mathcal{T}_{\sigma}(\phi)=3$, unless ${\mathrm{(i)}}$ the axis of rotation is horizontal and $0<\alpha<\pi$, when $\dim \mathcal{T}_{\sigma}(\phi)=4$, or ${\mathrm{(ii)}}$ the axis of rotation is neither horizontal nor vertical and $\alpha=\pi$, when $\dim \mathcal{T}_{\sigma}(\phi)=5$, or ${\mathrm{(iii)}}$ the axis of rotation is horizontal and $\alpha=\pi$, when $\dim \mathcal{T}_{\sigma}(\phi)=6$, or ${\mathrm{(iv)}}$ the axis of rotation is vertical and $\alpha=\pi$, when $\dim \mathcal{T}_{\sigma}(\phi)=7$.
\end{lem}

\begin{proof}
As in the proof of Lemma~\ref{onec},  $T\in \mathcal{T}_{\sigma}(\phi)$ if and only if
\begin{equation}\label{eq:cond3b}
(\phi p^{(1)})_j - p^{(2)}_j=0,\quad
(\phi p^{(2)})_j - p^{(1)}_j=0,\quad{\text{and}}\quad
(\phi p^{(3)})_j +p^{(1)}_j+p^{(2)}_j+p^{(3)}_j=0,
\end{equation}
for $i=1,2,3$ and $j=1,2$, where in the third equation we have used (\ref{cendep}) to write $p^{(4)}_j$ in terms of $p^{(i)}_j$, $i=1,2,3$. Let $M_2$ be the $6 \times 9$ coefficient matrix of this system of six equations in nine variables, under the same convention as in the proof of Lemma~\ref{onec}.

From (\ref{eq:cond3b}), we have
$$M_2=\left( \begin{array}{ccc}
    A & -I & 0 \\
    -I & A & 0\\
    I & I & A+I
  \end{array} \right),$$
where $A$ and $I$ are given by (\ref{AI}).

The matrix $M_2$ has rank at most six, so $\dim \mathcal{T}_{\sigma}(\phi)\ge 9-6=3$.

Let $J_2$ be the ideal generated by all $6 \times 6$ minors of $M_2$ together with the polynomial $\tau=a^2+b^2+c^2+d^2-1$. A Gr\"obner basis $G_2$ of $J_2$ with respect to the lexicographic ordering $a \succ b \succ c \succ d$ is
\begin{eqnarray*}
G_2&=&\{a^2 d^2(d^2 - 1), a^2 c d (d^2 - 1), a^2 c^2 d , a^2 b d (d^2 - 1),
a^2 b c d , a^2 (b^2 - c^2) d , a^3 c d ,
a^3 b d, \tau\}.
\end{eqnarray*}

From this we see that if $\dim \mathcal{T}_{\sigma}(\phi)>3$, then $a=0$ or $d=0$. (Note that $d=\pm 1$ implies $a=b=c=0$.)  One can check that the rank of $M_2$ is $5$, $4$, $3$ or $2$, when $a\neq 0$ and $d=0$, or when $a=0$ and $d\neq 0,\pm 1$, or when $a=d=0$, or when $a=0$ and $d=\pm 1$, respectively.  This yields $\dim \mathcal{T}_{\sigma}(\phi)$ for cases (i), (ii), (iii), and (iv) in the statement of the lemma.
\end{proof}

The geometry behind the previous lemma is straightforward using the analysis given after Lemma~\ref{onec} and bearing in mind the centroid condition (\ref{cendep}).  We omit the details. Note that case (iv), when $\phi$ is a rotation by $\pi$ about the $z$-axis, was already mentioned in the introduction.

\section{Three-cycle: $\sigma=(2,3,1,4)$\label{se:threec}}

\begin{lem}\label{threec}
Let $\phi \neq \id$ be a rotation by angle $\alpha$ and let $\sigma=(2,3,1,4)$.  Then $\dim \mathcal{T}_{\sigma}(\phi)=3$, unless ${\mathrm{(i)}}$ the axis of rotation is horizontal, in which case $\dim \mathcal{T}_{\sigma}(\phi)=4$, or ${\mathrm{(ii)}}$ the axis of rotation is vertical and $\alpha=2\pi/3$, in which case $\dim \mathcal{T}_{\sigma}(\phi)=5$.
\end{lem}

\begin{proof}
As in the proof of Lemma~\ref{onec},  $T\in \mathcal{T}_{\sigma}(\phi)$ if and only if
\begin{equation}\label{eq:cond3c}
(\phi p^{(1)})_j - p^{(2)}_j=0,\quad
(\phi p^{(2)})_j - p^{(3)}_j=0,\quad{\text{and}}\quad
(\phi p^{(3)})_j -p^{(1)}_j=0,
\end{equation}
for $i=1,2,3$ and $j=1,2$. Let $M_3$ be the $6 \times 9$ coefficient matrix of this system of six equations in nine variables, under the same convention as in the proof of Lemma~\ref{onec}.

From (\ref{eq:cond3c}), we have
$$M_3=\left( \begin{array}{ccc}
    A & -I & 0 \\
    0 & A & -I\\
    -I & 0 & A
  \end{array} \right),$$
where $A$ and $I$ are given by (\ref{AI}).

Then $M_3$ has rank at most six, so $\dim \mathcal{T}_{\sigma}(\phi)\ge 9-6=3$.

Let $J_3$ be the ideal generated by all $6 \times 6$ minors of $M_3$ together with the polynomial $\tau=a^2+b^2+c^2+d^2-1$. A Gr\"obner basis $G_3$ of $J_3$ with respect to the lexicographic ordering $a \succ b \succ c \succ d$ is
$$G_3=\{ d^2(4d^2-3)^2, cd(4d^2-3), bd(4d^2-3),d(b^2+c^2), \tau\}.$$

It follows that if $\dim \mathcal{T}_{\sigma}(\phi)>3$, then either $d=0$ or $d=\pm \sqrt{3}/2$ and $b=c=0$ (and hence $a=\pm 1/2$).  One can check that the rank of $M_3$ is then either 5 or 4, respectively.  This yields $\dim \mathcal{T}_{\sigma}$ for cases (i) and (ii) in the statement of the lemma.
\end{proof}

Again, we comment on the geometry behind the previous lemma.    If the axis of rotation $l_u$ is horizontal, there are three degrees of freedom in choosing $p^{(1)}$. Then $p^{(2)}$ and $p^{(3)}$ must lie in the vertical plane $H$ containing $p^{(1)}$ and orthogonal to $l_u$.  Moreover, $p^{(3)}$ must lie in the line obtained by rotating the vertical line through $p^{(1)}$ by $-\alpha$ around $l_u$ (so that $\phi p^{(3)}$ and $p^{(1)}$ have the same projection on the $xy$-plane). This is another degree of
freedom. Similarly,  $p^{(2)}$ must lie in the line obtained by rotating
the vertical line through $p^{(3)}$ by $-\alpha$ around $l_u$ (so that $\phi p^
{(2)}$ and $p^{(3)}$ have the same projections on the $xy$-plane).  But $p^{(2)}$ must also lie in the vertical line through $\phi p^{(1)}$ (so that $\phi p^{(1)}$ and $p^
{(2)}$ have the same projections on the $xy$-plane). This means that $p^{(2)}$ is  determined by the positions of $p^{(1)}$ and $p^{(3)}$ and so there are only four degrees of freedom in this case.

If the axis of rotation is the $z$-axis and the angle of rotation $\alpha=2\pi/3$, there are three degrees of freedom in choosing $p^{(1)}$ and a further one degree of freedom for each of $p^{(2)}$ and $p^{(3)}$, since their heights may be different from that of $p^{(1)}$ and only their horizontal positions are determined. Thus there are five degrees of freedom in all, corresponding to case (ii) in the statement of the lemma.

In the general case, there are three degrees of freedom in choosing $p^
{(1)}$, after which the other points are determined. Indeed, $p^{(2)}$ must lie
on the vertical line $l_1$, say, through $\phi p^{(1)}$, and $p^{(3)}$ must lie on the
vertical plane through the line obtained by rotating $l_1$ by $\alpha$ around $l_u$ (so that $\phi p^{(2)}$ and $p^{(3)}$ have the same projection on the $xy$-plane). Moreover, $p^{(3)}$ must also lie on the line obtained by rotating the vertical line through $p^ {(1)}$ by $-\alpha$ around $l_u$ (so that $\phi p^{(3)}$ and $p^{(1)}$ have the same projection on the $xy$-plane). Hence, in the general case, $p^{(3)}$ is determined and consequently also $p^{(2)}$.

\section{Four-cycle: $\sigma=(2,3,4,1)$\label{se:fourc}}

\begin{lem}\label{fourc}
Let $\phi \neq \id$ be a rotation by angle $\alpha$ and let $\sigma=(2,3,4,1)$.  Then $\dim \mathcal{T}_{\sigma}(\phi)=3$, unless ${\mathrm{(i)}}$ the axis of rotation is neither horizontal nor vertical and $\alpha=\pi$, in which case $\dim \mathcal{T}_{\sigma}(\phi)=4$, or ${\mathrm{(ii)}}$ the axis of rotation is vertical and either $\alpha=\pi/2$ or $\alpha=\pi$, in which case $\dim\mathcal{T}_{\sigma}(\phi)=5$.
\end{lem}

\begin{proof}
As in the proof of Lemma~\ref{onec},  $T\in \mathcal{T}_{\sigma}(\phi)$ if and only if
\begin{equation}\label{eq:cond3d}
(\phi p^{(1)})_j - p^{(2)}_j=0,\quad
(\phi p^{(2)})_j - p^{(3)}_j=0,\quad{\text{and}}\quad
(\phi p^{(3)})_j +p^{(1)}_j+p^{(2)}_j+p^{(3)}_j=0,
\end{equation}
for $i=1,2,3$ and $j=1,2$, where in the third equation we have used (\ref{cendep}) to write $p^{(4)}_j$ in terms of $p^{(i)}_j$, $i=1,2,3$. Let $M_4$ be the $6 \times 9$ coefficient matrix of this system of six equations in nine variables, under the same convention as in the proof of Lemma~\ref{onec}.

From (\ref{eq:cond3d}), we have
$$M_4=\left( \begin{array}{ccc}
    A & -I & 0 \\
    0 & A & -I\\
    I & I & A+I
  \end{array} \right),$$
where $A$ and $I$ are given by (\ref{AI}).

Since $M_4$ has rank at most six, we obtain $\dim \mathcal{T}_{\sigma}(\phi)\ge 9-6=3$.

Let $J_4$ be the ideal generated by all $6 \times 6$ minors of $M_3$ together with the polynomial $\tau=a^2+b^2+c^2+d^2-1$. A Gr\"obner basis $G_4$ of $J_4$ with respect to the lexicographic ordering $a \succ b \succ c \succ d$ turns out to be
\begin{eqnarray*}
G_4&=&\{ a^2 (2 d^2 - 1)^2, a^2 c  (2 d^2 - 1), a^2 b (2 d^2 - 1),
a^2 (b^2 + c^2) , a (d^2 - 1) (2 d^2 - 1)^2,\\
& & a c (2 d^2 -1), a b (2 d^2 - 1), a (b^2 + c^2), \tau \}.
\end{eqnarray*}

Consequently, if $\dim \mathcal{T}_{\sigma}(\phi)>3$, then either $a=0$ or $d=\pm 1/\sqrt{2}$ and $b=c=0$ (and hence $a=\pm 1/\sqrt{2}$).  It can be verified that the rank of $M_4$ is 5 if $a=0$ and $d\neq 0,\pm 1$ or 4 if either $a=0$ and $d=\pm 1$ or $a=d=\pm 1/\sqrt{2}$ and $b=c=0$.  This yields $\dim \mathcal{T}_{\sigma}$ for cases (i) and (ii) in the statement of the lemma.
\end{proof}

Regarding the previous lemma, suppose that the axis of rotation $l_u$ is not horizontal or vertical.  If the angle of rotation $\alpha=\pi$, there are two degrees of freedom for choosing $p^{(1)}$ in the vertical plane containing $l_u$, after which $p^{(2)}$ and $p^{(3)}$ can be chosen anywhere in the vertical line containing $p^{(1)}$, making four degrees of freedom in all.  This corresponds to case (i) in the statement of the lemma.

If $l_u$ is the $z$-axis and $\alpha=\pi/2$ or $\alpha=\pi$, there are three degrees of freedom for choosing $p^{(1)}$.  After this the horizontal positions of $p^{(2)}$ and $p^{(3)}$ are determined but their heights are arbitrary, giving a total of five degrees of freedom.  This deals with case (ii) in the statement of the lemma.

It remains to explain the generic case.  It is clear that there are no solutions when $l_u$ is the $z$-axis unless $\alpha=\pi/2$ or $\alpha=\pi$, and it is easy to see that if $l_u$ is horizontal, then we have three degrees of freedom in choosing $p^{(1)}$, after which the other points are determined.  Suppose, then, that $l_u$ is neither vertical nor horizontal and $0<\alpha<\pi$. For $i=2,3,4$, let $n^{(i)}=(p^{(1)}+p^{(i)})/2$ be the midpoint of the edge $[p^{(1)},p^{(i)}]$. Using (\ref{cendep}), we obtain
\begin{equation}\label{Defwi}
    \begin{array}{c}
      n^{(2)}= (p^{(1)}+p^{(2)}-p^{(3)}-p^{(4)})/4,\\
      n^{(3)}= (p^{(1)}-p^{(2)}+p^{(3)}-p^{(4)})/4, \\
      n^{(4)}= (p^{(1)}-p^{(2)}-p^{(3)}+p^{(4)})/4.
    \end{array}
\end{equation}
Notice from (\ref{Defwi}) that the points $n^{(i)}$ determine $T$ via the equations $p^{(1)}=n^{(2)}+n^{(3)}+n^{(4)}$, $p^{(2)}=n^{(2)}-n^{(3)}-n^{(4)}$,
$p^{(3)}=-n^{(2)}+n^{(3)}-n^{(4)}$, and $p^{(4)}=-n^{(2)}-n^{(3)}+n^{(4)}$.
We shall therefore focus on the degrees of freedom in specifying $n^{(i)}$, $i=2,3,4$.

To this end, suppose that
\begin{equation}\label{FC}
    \phi p^{(i)}=p^{(i+1)}+\mu_ie_3\, ,
\end{equation}
for $i=1,\dots,4$, where $\mu_i \in \R$, indices are taken modulo 4, and $e_3$ denotes the unit vector in the direction of the positive $z$-axis. From (\ref{Defwi}) and (\ref{FC}), we get
\begin{equation}\label{Phi(wi)}
\begin{array}{c}
\phi n^{(2)}= -n^{(4)}+(\mu_1+\mu_2)e_3/2, \\
  \phi n^{(3)} = -n^{(3)}+ (\mu_1+\mu_3)e_3/2,\\
  \hspace{-.12in}\phi n^{(4)} = n^{(2)}+(\mu_1+\mu_4)e_3/2.
\end{array}
\end{equation}

For $i=2,3,4$, let $C_i$ be the circle contained in a plane orthogonal to $u$, with center $c^{(i)}$ on $l_u$, and containing $n^{(i)}$.  We claim that the position of $c^{(3)}$ on $l_u$ alone determines that of $n^{(3)}$, and hence there is only one degree of freedom in choosing $n^{(3)}$.  To see this, note that $-C_3$ is the circle orthogonal to $u$, with center $-c^{(3)}$ on $l_u$, and containing $-n^{(3)}$.  Let $u^{\perp}$ be the plane through the origin orthogonal to $u$ and let $\pi$ denote the parallel projection in the direction $e_3$ onto $u^{\perp}$.  Then $\pi C_3$ and $\pi(-C_3)$ are circles in $u^{\perp}$ with centers $\pi c^{(3)}$ and $\pi(-c^{(3)})$ on $\pi l_u$.  See Figure~\ref{fi:picture3}.  By the second equation in (\ref{Phi(wi)}), $\pi(\phi n^{(3)})=\pi(-n^{(3)})$.  Since $\pi(\phi n^{(3)})$ lies on $\pi C_3$, we see that $\pi C_3$ and $\pi(-C_3)$ intersect at $\pi n^{(3)}$ and $\pi(-n^{(3)})$.  Now $\pi n^{(3)}$ and $\pi(-n^{(3)})$ lie on a line through the origin, so this line is orthogonal to $\pi l_u$.

\begin{figure}[htb]
\epsfig{file=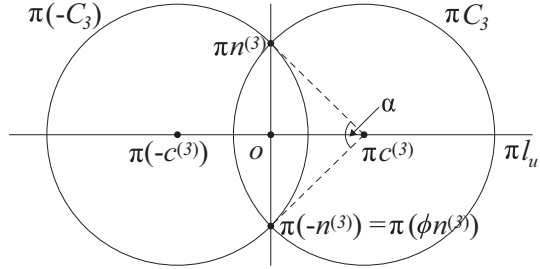, bb=69 19 342 158, clip=,
width=3in}
\caption{Geometry in the plane $u^{\perp}$ after parallel projection $\pi$.}
\label{fi:picture3}
\end{figure}

Therefore we have $\pi n^{(3)}=\lambda(u\times e_3)/\|u\times e_3\|$, for some real $\lambda$, because $\pi n^{(3)}$ is orthogonal to both $u$ and $e_3$. The angle between $\pi n^{(3)}$ and $\pi(-n^{(3)})=\pi(\phi n^{(3)})$ at $\pi c^{(3)}$ is $\alpha$, so
$$
 \lambda=\|\pi n^{(3)}\|=\|\pi c^{(3)}\|\tan(\alpha/2).
$$
It follows that once the position of $c^{(3)}$ on $l_u$ is specified, we know $\pi c^{(3)}$ and therefore $\lambda$ and hence $\pi n^{(3)}$. We also know the radius $\left(\|\pi c^{(3)}\|^2+\|\pi n^{(3)}\|^2\right)^{1/2}$ of $\pi C_3$, which equals that of $C_3$.  From this and $\pi n^{(3)}$, the position of $n^{(3)}$ is determined.  This proves the claim.

Next, we consider $n^{(2)}$ and $n^{(4)}$.  The first and third equations in (\ref{Phi(wi)}) tell us that
\begin{equation}\label{pisnew}
\pi(\phi n^{(2)})=\pi(-n^{(4)})\quad{\text{and}}\quad \pi(\phi n^{(4)})=\pi n^{(2)}.
\end{equation}
Identify $u^\perp$ with the complex plane $\C$ in such a way that $\pi l_u$ is the real axis.  Then since $\pi c^{(2)}$ and $\pi c^{(4)}$ lie on $\pi l_u$, they are real. Let $\omega = \exp(-\alpha i)$. Then by (\ref{pisnew}), we have
\begin{equation}\label{ComplexRelation}
    \begin{array}{c}
      (\pi n^{(2)}-\pi c^{(2)})\omega = -\pi n^{(4)}-\pi c^{(2)}\, , \\
      (\pi n^{(4)}-\pi c^{(4)})\omega = \pi n^{(2)}-\pi c^{(4)}\, .
    \end{array}
\end{equation}
If $\omega^2+1\neq 0$, we can solve the linear system (\ref{ComplexRelation}) for $\pi n^{(2)}$ and $\pi n^{(4)}$ in terms of $\pi c^{(2)}$ and $\pi c^{(4)}$.  Therefore once the positions of $c^{(2)}$ and $c^{(4)}$ on $l_u$ are specified (for which there are two degrees of freedom), we know $\pi c^{(2)}$ and $\pi c^{(4)}$, hence $\pi n^{(2)}$ and $\pi n^{(4)}$.  As above, this allows the radii of the circles $C_2$ and $C_4$ to be determined, and then $n^{(2)}$ and $n^{(4)}$ are also determined.
Finally, if $\omega^2+1=0$, then $\alpha=\pm \pi/2$.  Then it is easy to see that for (\ref{pisnew}) to hold, we must have $\pi C_2=\pi C_4$.  This means that $l_u$ is vertical, which is not the case.

\begin{figure}[htb]
\epsfig{file=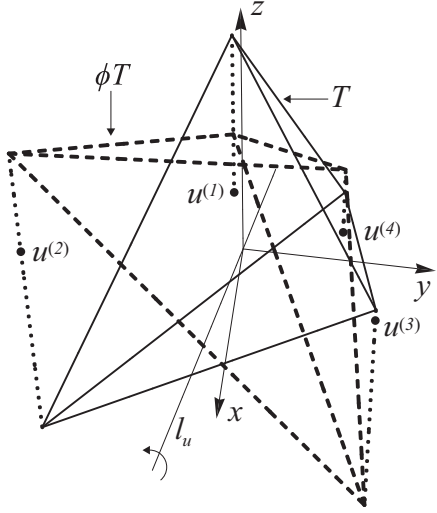, bb=117 188 370 434, clip=,
width=3in}
\caption{A four-cycle example (perspective view).  Dotted lines through the common projections $u^{(i)}$, $i=1,\dots,4$, of pairs of vertices of the two tetrahedra onto the $xy$-plane are vertical.}
\label{fi:4cycle}
\end{figure}

\begin{example}
For a specific example of the four-cycle situation, let $u=(1/\sqrt{2},0,1/\sqrt{2})$ and let $\phi$ be the rotation around $l_u$ by  $\pi/3$, so that $\phi$ has matrix
$$\left[\begin{array}{ccc}
3/4 & -\sqrt{3/8} & 1/4 \\
\sqrt{3/8} & 1/2 & -\sqrt{3/8}\\
1/4 & \sqrt{3/8} & 3/4
\end{array}\right].$$
Let $p^{(1)}=(-2, -3 + \sqrt{6}, 16 - 3 \sqrt{6})$, $p^{(2)} = (1, 3 - 4 \sqrt{6}, -19 + 3 \sqrt{6})$,
$p^{(3)} = (2, -3 + 3 \sqrt{6}, 8 - 3 \sqrt{6})$, and $p^{(4)}= (-1, 3, -5 + 3 \sqrt{6})$.
Then it is easy to check that the tetrahedron $T=\{p^{(1)},\dots,p^{(4)}\}$ is full dimensional, and
$\phi p^{(1)}=(1, 3 - 4 \sqrt{6}, 13 - 3 \sqrt{6})$, $\phi p^{(2)} = (2, -3+ 3\sqrt{6}), -20 + 3\sqrt{6})$,
$\phi p^{(3)} = (-1, 3, 11 - 3 \sqrt{6})$, and $\phi p^{(4)}= (-2, -3 + \sqrt{6}, -4 + 3 \sqrt{6})$.  The projections onto the $xy$-plane give the set
$U=\{u^{(1)},\dots,u^{(4)}\}$, where $u^{(1)}=(-2, -3 + \sqrt{6})$, $u^{(2)}=(1, 3 - 4 \sqrt{6})$,
$u^{(3)}=(2, -3 + 3 \sqrt{6})$, and $u^{(4)}=(-1, 3)\}$.
See Figure~1 for an illustration of this example.
\end{example}

\section{Main results}\label{se:main}

\begin{thm}\label{thm0}
For almost all tetrahedra $T$ in $\R^3$ with centroid at the origin, there does not exist a $\phi\neq \id\in \SO(3)$ such that $\pi_z\phi T=\pi_zT$.  Indeed, the exceptional set constitutes a finite union of subspaces, each of dimension at most seven, in $\R^9$.
\end{thm}

\begin{proof}
Let $T$ be a tetrahedron in $\R^3$ with centroid at the origin, and suppose that $\phi\neq \id\in \SO(3)$ is such that $\pi_z\phi T=\pi_zT$.  Then there is a $\sigma_0\in \Sym_4$ such that (\ref{eq:cond2}) holds with $\sigma=\sigma_0$ and hence
$$T\in \mathcal{T}_{\sigma_0}(\phi) \subset \bigcup_{\sigma\in \Sym_4}\mathcal{T}_{\sigma}(\phi).$$
By our conventions and Lemmas~\ref{ident}, \ref{onec}, \ref{twoc}, \ref{threec}, and~\ref{fourc}, the latter set is a finite union of subspaces of $\R^9$, each of which has dimension at most seven.  Therefore this set is of zero Lebesgue $9$-dimensional measure and the theorem is proved.
\end{proof}

\begin{ex}\label{untet}
We claim that a specific example of a full-dimensional tetrahedron satisfying Theorem~\ref{thm0} is $T=\{p^{(1)},\dots,p^{(4)}\}$, where
$$
  p^{(1)} = (1,0,0), \quad
  p^{(2)} = (1,1,0), \quad
  p^{(3)} = (2,1,2), \quad {\text{and}}\quad
  p^{(4)} = (4,-2,-2).
$$
To see this, observe that
\begin{equation}\label{new1}
\|p^{(1)}\|^2 = 1, \quad
  \|p^{(2)}\|^2 = 2, \quad
  \|p^{(3)}\|^2 = 9, \quad {\text{and}}\quad
  \|p^{(4)}\|^2 = 24,
\end{equation}
while
\begin{equation}\label{new2}
\|\pi_zp^{(1)}\|^2 = 1, \quad
  \|\pi_zp^{(2)}\|^2 = 2, \quad
  \|\pi_zp^{(3)}\|^2 = 5, \quad {\text{and}}\quad
  \|\pi_zp^{(4)}\|^2 = 20,
\end{equation}
By (\ref{eq:cond2}), we have
$$\|\pi_zp^{(\sigma(i))}\|=\|\pi_z\phi p^{(i)}\|\le \|\phi p^{(i)}\|=\|p^{(i)}\|,$$
for $i=1,\dots,4$.  Comparing (\ref{new1}) and (\ref{new2}), we see that the only possibility is that $\sigma=\id$. Since $T$ is full dimensional, our claim follows from Proposition~\ref{PHW}.
\end{ex}

From a practical point of view, perhaps the most important observation is that there are only 24 ways to label the points in a tetrahedron $T$ in $\R^3$.  If $T$ is full dimensional, then, for any particular such labeling, $\phi$ can be reconstructed by the method of Section~\ref{se:labeled}, or that of Robinson, Hemler, and Webber \cite{rhw-2000}, or symbolically using standard software, yielding at most 24 solutions for the rotation $\phi$.

We close with a remark illustrating how the uniqueness issues are
reflected within symbolic reconstruction methods.
While for each fixed permutation $\sigma$, there is a unique solution
for reconstructing a full-dimensional tetrahedron
(since fixing the permutation allows Proposition~\ref{PHW} to be applied), there may be more than one solution for lower-dimensional tetrahedra.  For example, consider the tetrahedron $T$ with vertices
$p^{(1)} = (-1,0,1)$, $p^{(2)} = (0,0,0)$, $p^{(3)} = (0,0,-2)$, and $p^{(4)} = (1,0,1)$, with $\pi_zT=U=\{(-1,0),(0,0),(1,0)\}$. Let $\sigma = (2,3)$ be the one-cycle that interchanges 2 and 3.  Solving symbolically, we obtain the Gr\"obner basis $\{ d,c,b^3-b,ab,a^2+b^2-1 \}$, which yields four distinct solutions for $(a,b,c,d)$.  However, these only result in two rotation matrices,
$$\left( \begin{array}{ccc}
    1 & 0 & 0 \\
    0 & 1 & 0 \\
    0 & 0 & 1
  \end{array} \right) \quad \text{ and } \quad
\left( \begin{array}{ccc}
    1 & 0 & 0 \\
    0 & -1 & 0 \\
    0 & 0 & -1
  \end{array} \right).
$$
Of course, these correspond to the two possible rotations $\phi$ such that $\pi_z\phi T=U$, namely, the identity and the rotation by $\pi$ about the $x$-axis.

\providecommand{\bysame}{\leavevmode\hbox to3em{\hrulefill}\thinspace}
\providecommand{\MR}{\relax\ifhmode\unskip\space\fi MR }
\providecommand{\MRhref}[2]{%
  \href{http://www.ams.org/mathscinet-getitem?mr=#1}{#2}
}
\providecommand{\href}[2]{#2}


\end{document}